\documentclass[12pt,a4paper,oneside]{amsart}
\usepackage[T1]{fontenc}	    
\usepackage[utf8]{inputenc}         
\usepackage{lmodern}		    
\usepackage{amsmath}
\usepackage{amsfonts}
\usepackage{amssymb,amsthm,amstext,amsopn}
\usepackage{xcolor}
\usepackage{enumerate}
\newtheorem{theorem}{Théorème}
\newtheorem{proposition}{Proposition}
\newtheorem{corollary}{Corollaire}
\newtheorem{lemma}{Lemme}

\newtheorem*{property}{Propriétés}

\newtheorem{remark}{Remarque}

\theoremstyle{definition}
\newtheorem*{definition}{Définition}

\newtheorem*{consequence}{Conséquence}

\usepackage[english,french]{babel}

\newcommand{\dis}{\displaystyle}

\newcommand{\xdownarrow}[1]
{%
{\left\downarrow\vbox to #1{}\right.\kern-\nulldelimiterspace}
}
\newcommand{\xuparrow}[1]
{%
{\left\uparrow\vbox to #1{}\right.\kern-\nulldelimiterspace}
}

\author{Hajer HMILI}
\author{Isabelle LIOUSSE}

\begin{document}

\address{ {\bf Hajer HMILI},  Institut supérieure des sciences appliquées et de technologie de Mateur, 7030 Mateur, Bizerte, Tunisie.
\emph {hajermido@yahoo.fr }}

\address{{\bf Isabelle LIOUSSE}, Laboratoire Paul Painlevé, Université de Lille, 59655 Villeneuve d'Ascq Cédex, France.  \emph {liousse@math.univ-lille1.fr}}

\title[Groupes de Brown-Thompson]{Nombre de classes de conjugaison d'éléments d'ordre fini dans les groupes de Brown-Thompson}

\begin{abstract} 

Nous étendons un résultat de Matucci (\cite{Mat}) sur le nombre de classes de conjugaison d'éléments d'ordre fini dans le groupe de Thompson $T$.  D'après \cite{Lio},  le groupe de  Brown-Thompson  $T_{r,m}$ ne contient pas  d'élément d'ordre $q$ lorsque $pgcd(m-1,q)$ ne divise pas $r$. Nous montrons que si  $pgcd(m-1,q)$  divise $r$ alors il y a exactement $\varphi(q). pgcd(m-1,q)$  classes de conjugaison d'éléments d'ordre $q$ dans $T_{r,m}$, où $\varphi$ est la fonction phi  d'Euler. Comme corollaire, nous obtenons que  le groupe de Thompson $T$ n'est isomorphe à  aucun des groupes $T_{r,m}$ avec $m\not=2$ et tout morphisme de $T$ dans $T_{r,m}$, avec $m\not=2$ et $r\not= 0$ $mod \ (m-1)$, est trivial.

\medskip

\noindent A{\textsc{bstract.}} We extend a result of Matucci (\cite{Mat}) on the number of conjugacy classes of finite order elements  in the Thompson group $T$. According to  \cite{Lio}, if $ gcd(m-1,q)$ is not a divisor of $r$ then there does not exist element of order $q$ in the Brown-Thompson group $T_{r,m}$. We show that if $ gcd(m-1,q)$ is a divisor of $r$ then there are exactly $\varphi(q). gcd(m-1,q)$ conjugacy classes of elements of order $q$ in $T_{r,m}$, where $\varphi$ is the Euler function phi.  As a corollary, we obtain  that the Thompson group $T$  is isomorphic to none of the groups $T_{r,m}$, for $m\not=2$ and any morphism from $T$ into  $T_{r,m}$, with $m\not=2$ and $r\not= 0$ $mod \ (m-1)$, is trivial.    \end{abstract}

\subjclass{20E45, 37E10, 37E15}
\keywords{Thompson's groups, Brown's groups, torsion elements, conjugacy classes, PL-homeomorphisms of the circle, isomorphisms}
\maketitle

\section{Introduction et définitions}

En 1965, R. Thompson découvrit  les premiers exemples de groupes $T\subset  V$ de présentation finie,  simples et infinis. Le groupe $T$ [resp. $V$] se représente comme groupe d'homéomorphismes [resp. échanges d'intervalles] affines par morceaux du cercle (voir \cite{CFP}, \cite{Ste}). En 1987, K. Brown (\cite{Bro}) a défini une famille  $T_{r,m} \subset V_{r,m}$ englobant $T$ et $V$  et les groupes $V_{r,m}$ sont isomorphes aux groupes $G_{r,m}$ de Higman (\cite{Hig}).

Plus précisément, soit  $r$ un entier strictement positif, on note $\mathbb S_r$ le cercle ${\mathbb R }/{r\mathbb Z}$ de longueur $r$. Le cercle de longueur $1$ est $\mathbb S_1$, nous le noterons plus classiquement $\mathbb S^1$.

\begin{definition}
Un homéomorphisme $f$ du  cercle $\mathbb S_r$ est 
{\em affine par morceaux} s'il existe une subdivision finie 
$0<a_1<a_2<\dots<a_p=r$ de l'intervalle $[0,r]$ et un relevé ${\tilde
  f}$ de $f$ à $\mathbb R$ tels que $\displaystyle {\tilde f}_{\vert [a_i,a_{i+1}] } (x) 
  = \lambda_i x + \beta_i, \quad \lambda_i, \beta_i\in \mathbb R.$

Les  points $a_i$ sont appelés  \emph{points de coupure} de~$f$ et les nombres  $\lambda_i$, {\em pentes de~$f$}. 

Le groupe des homéomorphismes affines par morceaux de $\mathbb S_r$ préservant l'orientation est 
noté $PL^+ (\mathbb S_r)$.
\end{definition} 
 
\begin{definition}Soient $r$ et $m\geq 2$ deux entiers strictement positifs.  On  définit le \textit {groupe de Brown-Thompson} $T_{ r,m}$  comme l'ensemble des éléments $f$ de $PL^+ (\mathbb S_r)$  tels que : 
\begin{itemize}
\item les pentes de $f$  appartiennent à $\langle m \rangle  =\{m^s, s\in \mathbb Z\}$.
\item les points de coupure  de $f$ appartiennent à $\mathbb Z [\frac{1}{m}]=\{N. m ^s \vert N, s \in \mathbb Z\}$, 
\item les images par $f$ des points de coupure  de $f$ appartiennent à  $\mathbb Z [\frac{1}{m}]$.
\end{itemize}
Le  \textit {groupe de  Thompson $T$} est $T_{1,2}$.
\end{definition}

De nombreux auteurs se sont intéressés aux invariants et à la question  d'isomorphicité pour ces groupes de type Thompson (\cite{BiSt}, \cite{Bri}, \cite{Bri2}, \cite{BrGu}, \cite{Bro}, \cite{Hig}, \cite{Lio}, \cite{Mat}, \cite{Ste} \ $\cdots$). Dans cet article, nous nous concentrons sur les obstructions à l'isomorphicité entre groupes $T_{r, m}$ issues des éléments d'ordre fini et de leurs classes de conjugaison. Le calcul du nombre de ces classes  fût effectué pour $G_{r,m}$ par Higman (\cite{Hig}, section 6),  pour $T$ par Matucci(\cite{Mat}) puis ultérieurement par Geoghegan-Varisco(\cite{GeVa}) et Fossas(\cite{Fos}).
Comme dans \cite{Mat} et \cite{GeVa}, nous utilisons la représentation comme groupe d'homéomorphismes affines par morceaux du cercle et disposons ainsi d'un invariant dynamique suplémentaire : le nombre de rotation de Poincaré.  Nous indiquons sa définition et ses premières propriétés (voir \cite{Her} ou \cite{KH}).

\begin{definition} Soit $f$ un homéomorphisme  du  cercle $\mathbb S_r$, on définit le \textit{nombre de rotation sur $\mathbb S_r$} de $f$  par : $\dis  \rho(f) = \lim\limits_{n\rightarrow \infty} ({\tilde
f^n (0)/ rn }  ) \ (mod \ 1)  \ \in \mathbb S^1.$
\end{definition}
Ce nombre ne dépend pas du choix du relevé $\tilde{f}$ et satisfait les  propriétés classiques  : 

\begin{property} \ 
 
\begin{itemize}
\item $\rho (R_{\alpha}) =\frac{\alpha}{r}$ où $R_{\alpha}(x)= x+\alpha \ (mod \  r)$,

\item  $\rho(f^n) = n \, \rho(f)$  pour tout $n\in \mathbb Z$, 

\item  si $f$ est d'ordre fini $q\in \mathbb N^{>1}$ alors $\rho(f) =\frac{p}{q}$ avec $p<q$ et $p\wedge q=1$,

\item  soit  $h:\mathbb S_r \rightarrow \mathbb S_{r'}$  un homéomorphisme préservant l'orientation, $\rho(h\circ f\circ h^{-1}) =\rho (f)$.
\end{itemize}
\end{property}

Commençons par cette observation  : tout élément d'ordre $q$ est conjugué dans $PL^+ (\mathbb S^1)$  à une rotation d'angle $\frac{p}{q}$ avec $p\wedge q = 1$, une conjugante est construite par moyennisation (voir par exemple \cite{KH}, Proposition 11.2.2). Comme deux rotations d'angles différents ne sont jamais $C^0$-conjuguées, le nombre de classes de conjugaison d'éléments d'ordre $q$ dans $PL^+ (\mathbb S^1)$ est exactement le nombre d'entiers $p<q$ premiers avec $q$ c'est à dire $\varphi(q)$ (la fonction phi d'Euler). 
Le Théorème 7.1.5 de  \cite{Mat} (voir aussi \cite{GeVa} et \cite{Fos})  exprime qu'il est encore vrai pour le groupe de Thompson $T$ : \textit{"dans $T$, tout rationnel de $\mathbb S^1$ est réalisé comme nombre de rotation d'une unique classe de conjugaison d'éléments d'ordre fini"}. 

Ici, nous établissons que cette propriété n'est plus satisfaite par les autres groupes $T_{r, m}$ : 
\begin{theorem}\label{thm:1}
Soient $r\geq 1$, $m\geq 2$ et  $q\geq 2$ des entiers.
\begin{enumerate}[A.]
\item Si pgcd $(m-1, q)$ ne divise pas $r$  alors il y a $0$ classes de conjugaison d'éléments d'ordre $q$ dans $T_{r,m}$.
\item Si  pgcd $(m-1, q)$ divise  $r$ alors il y a pgcd $(m-1, q)$ classes de conjugaison d'éléments d'ordre $q$  et de nombre de rotation $\frac{p}{q}$ dans $T_{r,m}$, pour tout entier $p$ premier avec $q$.
\item Si  pgcd $(m-1, q)$ divise  $r$ alors il y a $\varphi(q)$.pgcd$(m-1, q)$ classes de conjugaison d'éléments d'ordre $q$ dans $T_{r,m}$.
\end{enumerate} \end{theorem}
Nous en déduisons le
\begin{corollary}\label{coro :1} \ 
Tout rationnel de $\mathbb S^1$ est réalisé comme nombre de rotation d'une unique classe de conjugaison d'éléments d'ordre fini dans $T_{r,m}$ si et seulement si $m=2$.

Le groupe de Thompson $T$ n'est isomorphe à aucun des groupes $T_{r,m}$ avec $m\not=2$ et tout morphisme de $T$ dans $T_{r,m}$, avec $m\not=2$ et $r\not= 0$ $mod \ (m-1)$,  est trivial.


\end{corollary}
Ce corollaire contraste avec le résultat d'ubiquité de $F$ montré par Brin (\cite{Bri2}).
Notre approche diffère de celles de \cite{Mat}, \cite{GeVa} et \cite{Fos} au sens où elle est essentiellement basée sur un critère dû à  Bieri et Strebel \cite{BiSt}.

\section{Préliminaires}\label{Pre}

\subsection{Critère de Bieri-Strebel pour les groupes de Brown-Thompson.}

Nous reprennons ici le critère général de Bieri-Strebel déterminant à quelles conditions deux intervalles réels sont échangés par une application affine par morceaux avec points de coupure et pentes prescrites (voir Théorème A 4.1 de  \cite{BiSt}).

\begin{definition} \ 
\begin{itemize}
\item Un \textit{$m$-intervalle} est un intervalle réel dont les extrémités sont dans $\mathbb Z \left[\frac{1}{m}\right]$.
 \item Un homéomorphisme $f: I\rightarrow I'$ est dit  $PL_m$ s'il est affine par morceaux avec pentes dans $\langle m \rangle$  et points de coupure dans $\mathbb Z \left[\frac{1}{m}\right]$. 
\item Deux intervalles $I$ et $I'$ sont  dits \textit{$PL_m$-équivalents} s'il existe  un homéomorphisme $PL_m$ entre-eux.
\end{itemize}
\end{definition}

\begin{proposition} (\cite{BiSt}, \cite{Lio}) \label{BS} Deux $m$-intervalles $I$ et $I'$ sont $PL_m$-équivalents  si et seulement si $\vert I \vert -  \vert I' \vert\in ( m-1). \mathbb Z \left[\frac{1}{m}\right]$, où $\vert I \vert$ représente la longueur de l'intervalle $I$. \end{proposition}

\begin{proof}

Soient  $I=[a,c]$ et $I'=[a',c']$ avec  $a,a',c,c'\in \mathbb Z \left[\frac{1}{m}\right]$.    Supposons qu'il existe $f$ un homéomorphisme $PL_m$ entre $I$ et $I'$. Notons  $a=b_0<b_1 ...b_{n-1}<b_n =c$,  $ b_i\in \mathbb{Z} \left[\frac{1}{m}\right]$, les points de coupure de $f$ et $\lambda_i=m^{k_i}$, $k_i\in \mathbb{Z}$,  la pente de $f$  sur $[b_{i-1},b_{i}]$. Nous allons montrer que $\vert I \vert -  \vert I' \vert= (c-a)-(c'-a')\in ( m-1). \mathbb Z \left[\frac{1}{m}\right]$.

\smallskip

Comme $ c-a = \sum_i (b_i -b_{i-1})$ et  $c'-a' = \sum_i \lambda_i (b_i -b_{i-1})$,  on a 

$  (c-a) - (c'-a')  = \sum_i (1-\lambda_i)(b_i -b_{i-1})$ et 

$ (1-\lambda_i)= -(m-1)\sum_{p=0}^{k_i-1} m^{p}=(m-1)M_i$  avec $\displaystyle  M_i  \in\mathbb{Z}$ et finalement 

$(c-a) - (c'-a') = (m-1) \sum_i M_i (b_i -b_{i-1}) \in (m-1).\mathbb{Z} \left[\frac{1}{m}\right]$. 

\medskip

\noindent R\'eciproquement, supposons  $ \vert I \vert -  \vert I' \vert\in ( m-1).\mathbb{Z} \left[\frac{1}{m}\right]$ \ \ $(*)$.

Quitte à composer à la source  et au but  par des rotations d'angles convenables dans $\mathbb{Z} \left[\frac{1}{m}\right]$,  on peut supposer que $I=[0,b]$ et 
$I'=[0,b']$ avec $b,b' \in \mathbb{Z} \left[\frac{1}{m}\right]$ positifs.

La condition $(*)$ se traduit par le fait qu'il existe $a \in \mathbb{Z} \left[\frac{1}{m}\right]$
tel que $ b' = b + (m-1) a$ et il s'agit de construire un $PL_m$-homéomorphisme $f: [0,b] \to [0,b +  (m -1)a]$ pour tous $a,b\in \mathbb Z \left[\frac{1}{m}\right]$ avec $b\geq 0$ et  $b + (m-1) a\geq 0$. L'inverse d'un homéomorphisme $PL_m$ entre $m$-intervalles étant $PL_m$, on peut aussi supposer, quitte à échanger $b$ et $b'$, que $a\geq 0$.

\smallskip

{\bf Cas 1 :  $a<b$} ($0 \leq b-a\leq b$).  L'application  $f_0 : [0,b] \rightarrow [0, b+ (m-1)a]$  définie par   $$f_0(x) = \left\{ \begin{array}{ccc} & x &  {\text{ si }} x\in [0, b-a] \cr
& m(x -(b-a)) +(b-a) &  {\text{ si }} x\in [ b-a, b] 
\end{array} \right.$$ est l'homéomorphisme $PL_m$ cherché.

\smallskip

{\bf Cas 2 : $a\geq b$.}  Choisissons $p\in \mathbb N$ tel que $0\leq m^{-p} a <b$.  D'après le cas 1, il existe $f_0: [0,b] \rightarrow [0, b+ (m-1) m^{-p} a]$ ayant les propriétés requises. On définit alors $f_1 :  [0, b+ (m-1) m^{-p} a]\rightarrow [0, b+ (m-1) a]$ par $$f_1(x) = \left\{ \begin{array}{cc}  x &{\text{ si }} x\in [0, b] \cr
 m^p(x -b) + b  &{\text{ si }}  \ \ \ \ \ \ \ x\in [ b,  b+ (m-1) m^{-p} a]. 
\end{array} \right. $$ L'application cherchée est $f= f_1\circ f_0$. \end{proof}
\medskip

\begin{consequence}[Isomorphisme de Bieri-Strebel] \label{cons:1}
Soient $m>1$ un entier, si $r$ et $r'$ sont deux entiers positifs congrus modulo $m-1$ alors  les groupes $T_{r,m}$ et  $T_{r',m}$ sont isomorphes. Par suite tout $T_{r',m}$ est isomorphe à l'un des $m-1$ groupes  $T_{r,m}$, $r\in {1, ..., m-1}$.
\end{consequence}

\begin{remark}\label{rema :1}
Tous les intervalles dyadiques sont $PL_2$-équivalents et par suite tous les groupes $T_{r,2}$ sont isomorphes à $T$.
\end{remark}

\subsection{Nombres de rotation des éléments d'ordre fini.}

\begin{proposition} (\cite{Lio}) \label{prop :2} Soient  $m\geq 2$, $r\geq 1$ et    $q\geq 1$  des  entiers. 
\begin{enumerate}

\item Si le groupe  $T_{r,m}$ contient un élément d'ordre $q$ alors pour tout $p\in \mathbb{N}^*$, le groupe  $ T_{r,m}$ contient un élément d'ordre fini de nombre de rotation $\frac{p}{q}$.
 
\item Le groupe $T_{r,m}$ contient un élément d'ordre $q$ si et seulement si pgcd$(m-1,q)$ divise  $r$.
\end{enumerate}
\end{proposition}

\begin{proof} Nous supposons $q\geq 2$, pour $q=1$ le résultat est trivial. 

\textbf{Premier item.} Supposons qu'existe $f\in T_{r,m}$ d'ordre $q$.  On a $\rho(f)=\frac{n}{q}$ où  les entiers $n$ et $q$ sont premiers entre-eux. Par Bezout, il existe $u$ et $v$ entiers tels que $1 =un + vq$.

Soit $p\in \mathbb{N^*}$, on définit  un élément d'ordre fini de $T_{r,m}$ par $g=f^{up}$. On a  $ \rho(g)= \rho(f ^{up}) = up \rho(f) = \frac{upn}{q}=\frac{p(1-vq)}{q} = \frac{p}{q}-pv= \frac{p}{q} (mod 1)$.

\medskip

\textbf{Deuxième item.}

\textbf{Condition nécessaire}.  Supposons que $r$  soit un  multiple de $pgcd(q, m-1)$. D'après  Bezout, $r=uq + v(m-1)$, donc $r= uq $ modulo $(m-1)$. L'isomorphisme de
Bieri-Strebel implique que les   groupes $T_{uq, m }$ et  $T_{r, m}$
sont isomorphes. De plus, le groupe $T_{uq, m }$ contient  la  rotation $ x\mapsto x+u$ d'ordre $q$ et de nombre de rotation $\frac{1}{q}$.

\medskip

\textbf{Condition suffisante}.    Par  hypothèse et d'après le premier item, il existe  $f\in T_{r,m}$ d'ordre $q$ et  $\rho{(f)} =\frac{1} {q}$. Fixons  $\tilde f$ un relevé de  $f$ à  $\mathbb{R}$ et identifions $f$  à  $\tilde f (mod \  r)$.   

La $f$-orbite  de  $0$  est  ordonnée comme suit $0 <f(0) <....< f^{q-1}(0)<r$.

Les $q$ intervalles $I_i:=[f^{i-1}(0),f^{i}(0)]$, $i=1,\cdots,q$, sont tous $PL_m$-équivalents, donc d'après le critère de Bieri-Strebel,  $\vert I_i \vert =\vert I_1 \vert   \  mod \  (m-1).\mathbb Z \left[\frac{1}{m}\right]$ et $\vert I_1 \vert=f(0)$. Par conséquent $\displaystyle r= \vert I_1 \vert +\cdots+ \vert I_q \vert = qf(0) \  mod \ (m-1).\mathbb Z \left[\frac{1}{m}\right].$

On en déduit qu'il  existe des  entiers $u,v,s$ tels que  $r - qf(0) =(m-1)\frac {v} { m^s}$ \
et  \ $f(0)  = \frac {u} { m^s}$. Ainsi, 
$ m^s r -qu = (m-1)v$, autrement dit  $ m^s r = qu + (m-1)v$.  Ceci implique
que $ m^s r$ est un multiple du $pgcd(q, m-1)$.  Les entiers $(m-1)$ et  $m^s$ étant premiers  entre-eux, on conclut que $ r$ est un multiple de  $pgcd(q, m-1)$. \end{proof}

\subsection{Critère de Conjugaison $PL_m$}
 
\begin{proposition}\label{prop :3}
Soient $f_1$ et $f_2$ deux éléments de $T_{r,m}$ d'ordre fini $q$ et de nombre de rotation $\frac{1}{q}$, on note $f_i(0) =\frac{N_i}{m^{s_i}}$, $i=1,2$. Les propriétés suivantes sont équivalentes :
\begin{enumerate}
\item $f_1$ et $f_1$ sont $PL_m$-conjugués (dans $T_{r,m}$),
\item $N_2-N_1$ est un multiple de $m-1$,
\item $f_1(0)- f_2(0) \in (m-1) . \mathbb Z[\frac{1}{m}]$ (autrement dit,  les intervalles $[0,f_1(0)]$ et $[0,f_2(0)]$ sont $PL_m$-équivalents).
\end{enumerate} 
\end{proposition}

\begin{proof}
\begin{lemma}
Soit $a=\frac{N_a}{m^{s_a}}\in \mathbb Z[\frac{1}{m}]$, tout homéomorphisme $f\in T_{r,m}$ d'ordre $q$, de nombre de rotation $\frac{1}{q}$ et vérifiant $f(0)=a$ est $PL_m$-conjugué à la rotation $R_{N_a}$ de $\mathbb S_{qN_a}$.
\end{lemma}

\begin{proof}  L'intervalle $[0,r[$ s'écrit $\dis \bigcup_{i=1}^{q} I_i$, où  $I_i =[f^{i-1}(0), f^{i}(0)]$. 
  
On considère l'application affine $H_1 : I_1=[0,a] \rightarrow [0,N_a]$, $x \mapsto  m^{s_a} .x$ et on définit par récurrence  $H_i : I_i \rightarrow [i-1,i]$ par $H_{i+1} = R_1 \circ H_i \circ f^{-1}.$
 
 \smallskip 
 
On vérifie facilement que l'application $H : \mathbb S_r \rightarrow  \mathbb S_{qN_a}$ définie par $H_{\vert I_i} = H_i$ est un $PL_m$-homéomorphisme  qui conjugue $f$ à $R_{N_a}$.    \end{proof}  

\noindent $(1)\implies (2)$. D'après le lemme précèdent, $f_i$ est $PL_m$-conjuguée à la rotation $R_{N_i}$ de $\mathbb S_{qN_i}$. Il nous reste à étudier à quelles conditions deux telles rotations sont $PL_m$-conjuguées.

Soit $h :\mathbb S_{qN_1} \rightarrow \mathbb S_{qN_2}$ une $PL_m$-conjugaison entre $R_{N_1}$ et $R_{N_2}$, quitte à composer au but $h$ par la rotation $R_{-h(0)}$ de $\mathbb S_{qN_2}$, on peut supposer que $h(0)=0$. 

Les intervalles $[0,N_1]$ et $[0,N_2]$ étant $PL_m$-équivalents, l'entier $N_2-N_1 \in  (m-1).\mathbb Z [\frac{1}{m}]$ et par suite $N_2-N_1$ est un multiple de $m-1$.

\noindent $(2)\implies (1)$. Si $N_2-N_1$ est un multiple de $(m-1)$ alors  $[0,N_1]$ et $[0,N_2]$ sont $PL_m$-équivalents et on peut reprendre la preuve du lemme précédent avec pour $H_1$ l'homéomorphisme de Bieri-Strebel qui  échange ces 2 intervalles.

\noindent $(2) \Longleftrightarrow (3)$ résulte du calcul suivant :

$f_1(0)-f_2(0)=\frac{N_1}{m^{s_1}}-\frac{N_2}{m^{s_2}}= N_1(m^{-s_1} -1) -N_2(m^{-s_2} -1)  + (N_1-N_2)$

\hskip 4.7 truecm $\ =(N_1-N_2)  \ mod \  (m-1).\mathbb Z [\frac{1}{m}]$. \end{proof}

\section{Classes de conjugaison d'éléments d'ordre fini dans les groupes de Brown-Thompson.}

\begin{lemma}\label{lemm :2}
Soient $p$ et $q>1$ deux entiers premiers entre-eux  et $u>0$, $v$ entiers tels que $up+vq=1$. 

Deux éléments $f_1$ et $f_2$ de $T_{r,m}$ d'ordre fini $q$ et de nombre de rotation $\frac{p}{q}$ sont $PL_m$-conjugués si et seulement si $f_1^u$ et $f_2^u$  (de nombre de rotation $\frac{1}{q}$) sont $PL_m$-conjugués.
\end{lemma}

La preuve de ce lemme résulte du fait que la $PL_m$-conjugaison se transmet aux puissances et des généralités suivantes : 
  
On a $\rho ( f^{u})= \frac{up}{q}= \frac{1-vq}{q}=  \frac{1}{q}$ et $(f^u)^p = f^{1-vq} = f$ dès que $f$ est d'ordre fini $q$.
\begin{remark} \label{rema :2} Une conséquence de ce  lemme est qu'étant donnés  $p$ et $q$ deux entiers premiers entre-eux, le nombre de classes de conjugaison d'éléments d'ordre fini $q$ et de nombre de rotation $\frac{p}{q}$ ne dépend pas de $p$.
\end{remark}

\subsection{Classes de conjugaisons dans $T_{r,2}$.} \ 

Soit $q\in \mathbb N^*$, d'après l'invariance par conjugaison topologique du nombre de rotation et la remarque précédente, il suffit de déterminer le nombre de classes de conjugaison d'éléments d'ordre $q$ et de nombre de rotation  $\frac{1}{q}$.

La Proposition \ref{prop :2} indique que $T_{r,2}$ contient un élément $f_1$ d'ordre $q$ et de nombre de rotation  $\frac{1}{q}$. Tous les intervalles dyadiques étant $PL_2$-équivalents par la Remarque \ref{rema :1}, l'item (3) de la  Proposition \ref{prop :3} est vérifié pour tout autre  $f_2\in T_{r,2}$ d'ordre $q$ et de nombre de rotation  $\frac{1}{q}$. On en déduit qu'il y a exactement une classe de conjugaison d'éléments  d'ordre $q$ et de nombre de rotation  $\frac{1}{q}$ ; le résultat de Matucci en découle.

\subsection{Preuve du Théorème \ref{thm:1}}

\subsubsection{Preuve de l'item A} Il résulte directement de l'item (2) de la Proposition \ref{prop :2}.

\subsubsection{Preuve de l'item B} Par la Remarque \ref{rema :2}, il suffit d'établir le résultat pour $p=1$.

\begin{proposition}\label{prop :4}  Soient  $q\in \mathbb N^{>1}$ et $a=\frac{N_a}{m^{s_a}} \in \mathbb Z [\frac{1}{m}] \cap ]0,r[$. Les propriétés suivantes sont équivalentes
\begin{enumerate}
\item Il existe  $f\in T_{r, m}$ d'ordre $q$, vérifiant $f(0)=a$  et de nombre de rotation $\frac{1}{q}$. 
\item $r-qa \in (m-1) \mathbb Z [\frac{1}{m}]$.
\item $r-qN_a$ est un multiple de $m-1$.
\end{enumerate}
\end{proposition}
\begin{proof} \ 

\noindent $(1) \implies (2)$. Comme dans la preuve de la Proposition \ref{prop :2}, les $q$ intervalles $I_i:=[f^{i-1}(0),f^{i}(0)[$,  $i=1,\cdots,q$, sont tous $PL_m$-équivalents et forment une partition de $[0,r[$. Par conséquent $ r= q\vert I_1 \vert =qf(0)  \ mod \ (m-1) \mathbb Z [\frac{1}{m}]$.

\medskip

\noindent $(2) \implies (1)$. Supposons que $r-qa \in  (m-1) \mathbb Z [\frac{1}{m}]$.

\noindent \textbf{Cas 1 :} $r-qa \geq 0$. Considère l'homéomorphisme $f_a\in  PL_+(\mathbb S_r)$ défini par : 

\smallskip

$\left\{ \begin{array}{ccccc} 
& f_a(x)  = &x + a &\text{ si } &x \in [0, (q-2)a] \cr
& f_a(x)  = &h_{BS}(x) &\text{    si } &\ \ \ \ \ \ \  x \in [(q-2)a, (q-1)a] \cr
& f_a(x)  = &h_{BS}^{-1}(x) -(q-2)a &\text{ si } &x \in [(q-1)a, r],  \cr
\end{array}\right. $

\smallskip

\noindent où $h_{BS} : [(q-2)a, (q-1)a]  \rightarrow  [(q-1)a, r]$ est l'application $PL_m$  donnée par la Proposition \ref{BS}, son existence est garantie par $\vert [(q-1)a, r] \vert - \vert [(q-2)a, (q-1)a]\vert = r-qa \in (m-1) \mathbb Z [\frac{1}{m}]$. On vérifie facilement que $f_a(0)=a$ et $f_a$ est d'ordre $q$.

\medskip

\noindent \textbf{Cas 2 :} $r-qa < 0$, choisissons $p \in \mathbb N$ de sorte que $ r - q(m^{-p} a) \geq 0$.

On considère $H_{BS} \in T_{r,m}$ tel que $H_{BS} ([0,a])=[0,m^{-p} a] $, son existence est assurée par le critère de Bieri-Strebel, puisque $m^{-p} a-a = (m^{-p} -1) a \in  (m-1) \mathbb Z [\frac{1}{m}]$. 

Finalement, l'application $PL_m$ définie par  $\displaystyle f_a =H_{BS}^{-1} \circ  f_ {m^{-p} a} \circ  H_{BS}$ est d'ordre $q$ et satisfait $f_a(0)=a$.

\medskip

\noindent $(2) \Longleftrightarrow (3)$ est conséquence du calcul suivant : 

$\displaystyle r-qa=r-q\frac{N_a}{m^{s_a}}=\frac{r m^{s_a}-qN_a}{m^{s_a}} \in (m-1)\mathbb Z [\frac{1}{m}] \Longleftrightarrow $

$r m^{s_a}-qN_a = r (m^{s_a}-1) + (r-qN_a)$ est un multiple de $m-1$ $\Longleftrightarrow $

$r-qN_a$ est un multiple de $m-1$.  \end{proof}

Nous pouvons maintenant calculer le nombre de classes de conjugaison d'éléments d'ordre $q$ de nombre de rotation $\frac{1}{q}$ dans $T_{r,m}$. D'après les Propositions \ref{prop :3} et  \ref{prop :4},   cette quantité est égale au nombre de classes modulo $(m-1)$  d'entiers $N$ tels que $r-qN $ est un multiple de $m-1$, nous affirmons que c'est $d=pgcd(m-1,q)$.

En effet, sous la condition $d$ divise $r$, le critère d'isomorphisme de Bieri-Strebel nous permet de supposer que $r=qu$. Posons $P= u-N$, on a $r-qN = q(u-N)=qP$, le problème se ramène à déterminer le nombre de classes modulo $(m-1)$ d'entiers $P$ tels que $qP$ est un multiple de $m-1$. 

Puisque $m-1 = m_0.d$ et $q=q_0.d$ avec $m_0\wedge q_0=1$, l'entier $qP $ est un multiple de $m-1$ si et seulement si $q_0P $ est un multiple de $m_0$ et donc si et seulement si $P$ est un multiple de $m_0$.  Par conséquent, il y a exactement $d$ tels entiers entre $0$ et $m-2$.

\subsubsection{Preuve de l'item C} D'après l'invariance par conjugaison topologique du nombre de rotation et la Remarque \ref{rema :2}, le nombre de classes de conjugaison d'éléments d'ordre $q$ dans $T_{r,m}$ est  $\varphi(q)$.pgcd$(m-1, q)$, si   pgcd $(m-1, q)$ divise  $r$ et $0$ si non. 

\section{Problèmes d'isomorphisme et de plongement entre groupes de Brown-Thompson. Preuve du Corollaire \ref{coro :1}}

\begin{proposition}\label{prop :5} Soit $m\in \mathbb N^{>1}$.
\begin{enumerate}
\item Parmi les groupes $T_{r,m}$, pour $0< r \leq m-1$,  seul le groupe $T_{m-1,m}$ contient des éléments d'ordre quelconque. Ainsi, il n'existe pas de morphisme injectif $T_{m-1,m}\rightarrow T_{r,m}$, pour $0< r <m-1$. 
\item Si $m_1-1$ possède un diviseur qui ne divise pas $m_2-1$ alors il n'existe pas de morphisme injectif  $T_{1, m_2} \rightarrow  T_{1, m_1}$. 
\end{enumerate} \end{proposition}

\smallskip

\begin{proof}  \ 

(1) Puisque pour tout $q\in \mathbb N^{>1}$,  pgcd$(m-1,q)$ est un diviseur de $m-1$, le groupe $T_{m-1,m}$ contient des éléments de tout ordre. Réciproquement, si $0<r < m-1$, pgcd$(m-1,m-1)=m-1$ ne divise pas $r$ et il n'existe pas d'élément d'ordre $m-1$ dans $T_{r,m}$.

\smallskip

(2) Soit $d_1$ un diviseur de $m_1-1$ ne divisant pas $m_2-1$. D'une part,  pgcd$(m_1-1,d_1)=d_1$ ne divise pas $r=1$ et $T_{1, m_1}$ ne contient pas d'élément d'ordre $d_1$. 
D'autre part, pgcd$(m_2-1,d_2)=1$ divise $r=1$ et $T_{1, m_2}$ contient des éléments d'ordre $d_1$. Par conséquent, il n'existe pas de morphisme injectif  $T_{1, m_2} \rightarrow  T_{1, m_1}$.\end{proof}

{\bf Preuve du Corollaire \ref{coro :1}.}
Comme $T$ est simple, tout morphisme de $T$  dans $T_{r,m}$ est injectif ou trivial. D'après l'item (1) de la Proposition \ref{prop :5},  si $T$ s'injecte dans  $T_{r,m}$,  pour $0< r \leq m-1$, alors $r=m-1$. De plus,  $T_{m-1,m}$ contient $m-1$ classes de conjugaison d'éléments d'ordre $m-1$ et de nombre de rotation $\frac{1}{m-1}$ alors que $T$ n'en contient qu'une. Ces groupes ne sont isomorphes que lorsque $m-1=1$, correspondant au seul cas où tout rationnel  se réalise comme nombre de rotation d'une unique classe de conjugaison d'éléments d'ordre fini.
\bibliographystyle{alpha}
\bibliography{ordre_fini}
\end{document}